\documentclass[leqno,12pt]{article}
\usepackage{amsfonts,amsmath,amssymb}
\usepackage{amsthm}
\usepackage{enumerate}
\usepackage{graphics}

\usepackage{float}%
\usepackage{hyperref}%
\newcommand{\Ind}	{\operatorname{Ind}}

\newcommand{\scirc}{\raise2pt\hbox{${}_\circ$}}

\newcommand{\cL}[1]{\mathcal{L}^{\textnormal{conf}}_{#1}}
\newcommand{\hL}[1]{\mathcal{L}^{\textnormal{hol}}_{#1}}
\newcommand{\hpi}[1]{\pi^{\textnormal{hol}}_{#1}} 
\newcommand{\aRC}[2]{\textnormal{RC}^a_{#1, #2}} 
\newcommand{\adRC}[2]{\mathcal{RC}^a_{#1, #2}} 

\newcommand{\Rest}[1]{\textnormal{Rest}_{#1}} 

\newcommand{\KT}[2]
           {\widetilde{K}_{{#1},{#2}}^{\mathbb{T}}}

    \makeatletter
    \let\@fnsymbol\@alph
    \makeatother

\newtheorem{thmalph}{Theorem}

\newtheorem{theorem}{Theorem}[section]
\newtheorem{proposition}[theorem]{Proposition}

\newtheorem{lemma}[theorem]{Lemma}

\theoremstyle{remark}
\newtheorem{remark}[theorem]{Remark}
\theoremstyle{definition}
\newtheorem{example}[theorem]{Example}

\newtheorem{problem}[theorem]{Problem}

\newtheorem{thmalpha}{Theorem} 

\newtheorem{thmpalpha}{Theorem} 

\newtheorem{thmppalpha}{Theorem} 

\newtheorem{thmpppalpha}{Theorem} 

\newtheorem{question}[thmalpha]{Question}
\newtheorem{pquestion}[thmpalpha]{Question}
\newtheorem{ppquestion}[thmppalpha]{Question}
\newtheorem{pppquestion}[thmpppalpha]{Question}

\numberwithin{equation}{section}

\begin{document}
\title{Vector-valued covariant differential operators
for the M\"obius transformation}

\author{Toshiyuki KOBAYASHI,\footnote{Kavli IPMU (WPI)
and Graduate School of Mathematical Sciences, 
The University of Tokyo, 3-8-1 Komaba, Meguro, Tokyo, 153-8914 Japan;
E-mail address: \texttt{toshi@ms.u-tokyo.ac.jp},} \;
Toshihisa KUBO,\thanks{Graduate School of Mathematical Sciences, 
The University of Tokyo, 3-8-1 Komaba, Meguro, Tokyo, 153-8914 Japan;
E-mail address: \texttt{toskubo@ms.u-tokyo.ac.jp},}\;
Michael PEVZNER\thanks{
Laboratoire de Math\'ematiques de Reims,
Universit\'e de Reims-Champagne-Ardenne,
FR 3399 CNRS, F-51687, Reims, France;
E-mail address: \texttt{pevzner@univ-reims.fr}.}}

\date{} 

\maketitle

\begin{abstract}
We obtain a family of functional identities satisfied by vector-valued
functions of two variables and their geometric inversions. For this we introduce particular differential operators of arbitrary order
attached to Gegenbauer polynomials. These differential operators are
symmetry breaking for the pair of Lie groups $(SL(2,\mathbb C), SL(2,\mathbb R))$ that arise from conformal geometry.

\end{abstract}


\medskip
\noindent
\textit{%
Key words and phrases\/}: 
branching law, reductive Lie group,
symmetry breaking,
conformal geometry,
Verma module, F-method,
time reversal operator.

\setcounter{tocdepth}{2}
\tableofcontents

\section{A family of vector-valued functional identities}\label{sec:1}

Given a pair of functions $f$, $g$ on $\mathbb{R}^2\backslash \{(0, 0)\}$, 
we consider 
a $\mathbb{C}^2$-valued function 
$\vec{F}:= \begin{pmatrix}f \\g \end{pmatrix}$.
Define its ``twisted inversion" 
$ \vec{F}_\lambda$
with parameter $\lambda \in \mathbb{C}$ by
\begin{equation}\label{eqn:invF}
\vec{F}^\vee_\lambda(r\cos \theta, r\sin \theta)
:=r^{-2\lambda} 
\begin{pmatrix}
\cos 2\theta & -\sin 2\theta\\
\sin 2\theta & \cos 2\theta
\end{pmatrix}
\vec{F}\bigg(\frac{-\cos\theta}{r}, \frac{\sin\theta}{r} \bigg).
\end{equation}
Clearly, $\vec{F} \mapsto \vec{F}^\vee_\lambda$
is involutive, namely, 
$(\vec{F}^{\vee}_\lambda)^\vee_\lambda
= \vec{F}$.

\vskip 1pc
A pair of differential operators $\mathcal{D}_1$, 
$\mathcal{D}_2$ on $\mathbb{R}^2$
yields a linear map
\begin{equation*}
\mathcal{D}: C^\infty(\mathbb{R}^2) \oplus C^\infty(\mathbb{R}^2)
\to C^\infty(\mathbb{R}),
\quad (f,g) \mapsto (\mathcal{D}_1f)(x,0) + (\mathcal{D}_2g)(x,0).
\end{equation*}
We write
\begin{equation*}
\mathcal{D}:=\Rest{y=0} \circ (\mathcal{D}_1, \mathcal{D}_2).
\end{equation*}

Our main concern in this article is the following:

\begin{question} \label{quest:A}
(1) For which parameters $\lambda, \nu \in \mathbb{C}$,
do there exist differential operators $\mathcal{D}_1$ 
and $\mathcal{D}_2$ on $\mathbb{R}^2$
with the following properties? 
\begin{itemize}
\item $\mathcal{D}_1$ and $\mathcal{D}_2$ have constant coefficients.
\item For any $\vec{F} \in C^\infty(\mathbb{R}^2) \oplus C^\infty(\mathbb{R}^2)$,
the functional identity
\end{itemize}
\begin{equation}\label{eqn:A}
(\mathcal{D} \vec{F}^\vee_\lambda)(x) 
= |x|^{-2\nu}(\mathcal D\vec{F})\left(-\frac{1}{x}\right),
\quad
\text{for $x \in \mathbb{R}^{\times}$} \tag{$\mathcal{M}_{\lambda,\nu}$}
\end{equation}
holds, where 
$\mathcal{D} = \Rest{y=0} \circ 
(\mathcal{D}_1, \mathcal{D}_2)$.
\vskip 0.05in

\noindent
(2) Find an explicit formula of such 
$\mathcal{D} \equiv \mathcal{D}_{\lambda, \nu}$
if exists.

\end{question}

Our motivation will be explained in Section \ref{sec:2} by giving
three equivalent formulations of Question \ref{quest:A}.
Here are some examples of the 
operators $\mathcal{D}_{\lambda, \nu}$ satisfying \eqref{eqn:A}.

\begin{example} \label{ex:QA}
(0) $\nu = \lambda$:
\begin{equation*}
\mathcal{D}_{\lambda, \nu}:= \Rest{y=0} \circ 
\left(\text{id}, 0 \right),
\end{equation*}
namely,
\begin{equation*}
\mathcal{D}_{\lambda, \nu}\begin{pmatrix} f \\ g \end{pmatrix}(x)
=f(x,0)
\end{equation*}
satisfies \eqref{eqn:A} for $\nu = \lambda$.\\

\noindent
(1) $\nu = \lambda+1$:
\begin{equation*}
\mathcal{D}_{\lambda, \nu}:= \Rest{y=0} \circ 
\left(\frac{\partial}{\partial x}, \lambda \frac{\partial}{\partial y} \right),
\end{equation*}
namely,
\begin{equation*}
\mathcal{D}_{\lambda, \nu}\begin{pmatrix} f \\ g \end{pmatrix}(x)
=\frac{\partial f}{\partial x} (x, 0) + \lambda \frac{\partial g}{\partial y} (x,0)
\end{equation*}
satisfies \eqref{eqn:A} for $\nu = \lambda+1$.\\

\noindent
(2) $\nu = \lambda+2$:
\begin{equation*}
\mathcal{D}_{\lambda, \nu}:= \Rest{y=0} \circ 
\bigg(2 (2\lambda+1) \frac{\partial^2}{\partial x \partial y}, 
(\lambda-1)\frac{\partial^2}{\partial x^2} + (\lambda+1)(2\lambda + 1) 
\frac{\partial^2}{\partial y^2} \bigg),
\end{equation*}
namely,
\begin{equation*}
\mathcal{D}_{\lambda, \nu}\begin{pmatrix} f \\ g \end{pmatrix} (x)=
2(2\lambda+1) \frac{\partial^2 f}{\partial x \partial y }(x, 0) + 
(\lambda-1) \frac{\partial^2 g}{\partial x^2}(x, 0) + 
(\lambda+1)(2\lambda+1)\frac{\partial^2 g}{\partial y^2}(x, 0)
\end{equation*}
satisfies \eqref{eqn:A} for $\nu=\lambda+2$.
\end{example}

Given $\mathcal{D} = \Rest{y=0} \circ 
(\mathcal{D}_1, \mathcal{D}_2)$, define
\begin{equation}\label{eqn:Ddual}
\mathcal{D}^{\vee}:=\Rest{y=0} 
\circ (-\mathcal{D}_2, \mathcal{D}_1). 
\end{equation}
Clearly, $\mathcal{D}^\vee$ is determined only by
$\mathcal{D}$, and is independent of the choice of 
$\mathcal{D}_1$ and $\mathcal{D}_2$.
Proposition \ref{prop:Ddual} below
shows that the map $\mathcal{D} \mapsto \mathcal{D}^\vee$
is an automorphism of the set of 
the operators $\mathcal{D}$ such that 
\eqref{eqn:A} is satisfied.

\begin{proposition}\label{prop:Ddual}
If $\mathcal{D}$ satisfies \eqref{eqn:A} for all $\vec{F}$,
so does $\mathcal{D}^{\vee}$.
\end{proposition}

\begin{proof}
For $\vec{F}=\begin{pmatrix} f \\ g \end{pmatrix}$,
we set 
${}^\vee \! \vec{F}: = 
\begin{pmatrix} g \\ -f\end{pmatrix}$.
Then we have
\begin{equation}\label{eqn:FF}
^{\vee\vee}\!\vec{F}
=-\vec{F},
\quad
\mathcal{D}\left({}^\vee\! \vec{F} \right) = 
(\mathcal{D}^\vee)\vec{F},
\quad
({}^\vee\!\vec{F})^\vee_\lambda
={}^\vee\!(\vec{F}^\vee_\lambda).
\end{equation}

To see this we note that 
$w :=\begin{pmatrix} 0 & -1 \\ 1 & 0 \end{pmatrix}$
commutes with 
$\begin{pmatrix} 
\cos 2\theta & -\sin 2\theta \\
\sin 2\theta & \cos 2\theta
\end{pmatrix}$
and that $\mathcal{D}^\vee$ and ${}^\vee\! \vec{F}$ are expressed as
$\mathcal{D}^\vee = \mathcal{D} w^{-1}$ and ${}^\vee \!\vec{F} 
= w^{-1}\vec{F}$,
respectively. Therefore,
\begin{align*}
\left( \mathcal{D}^\vee \vec{F}^\vee_\lambda\right)(x)
&= \mathcal{D} \left({}^\vee \!\vec{F}\right)^\vee_\lambda(x)\\
&=|x|^{-2\nu}\left( \mathcal{D}\; {}^\vee \!\vec{F}\right)
\bigg( -\frac{1}{x} \bigg)\\
&=|x|^{-2\nu}\left( \mathcal{D}^\vee\vec{F}\right)
\bigg( -\frac{1}{x} \bigg),
\end{align*}
where the passage from the first line to the second one is justified by the fact that ${}^\vee\! \vec{F}$ satisfies \eqref{eqn:A}.
\end{proof}

In order to answer Question \ref{quest:A}
for general $(\lambda, \nu)$,
we recall that the Gegenbauer polynomial or
ultraspherical polynomial
$C^\alpha_\ell(t)$ 
is a polynomial in one variable $t$ of degree $\ell$ given by
\begin{equation*}
C^\alpha_\ell (t) =
\sum_{k=0}^{[\frac{\ell}{2}]} (-1)^k 
\frac{\Gamma(\ell - k + \alpha)}{\Gamma(\alpha)\Gamma(\ell-2k + 1) k!}
(2t)^{\ell - 2k},
\end{equation*}
where $[s]$ denotes the greatest integer that does not exceed $s$.
Following \cite{KP}, we inflate $C^\alpha_\ell(t)$ to a polynomial of two variables by
\begin{equation}\label{eqn:Gtwo}
C^\alpha_\ell(s,t):= s^{\frac{\ell}{2}} C^\alpha_\ell 
\bigg(\frac{t}{\sqrt{s}}\bigg).
\end{equation}
By formally substituting $-\frac{\partial^2}{\partial x^2}$
and $\frac{\partial}{\partial y}$ to $s$ and $t$
in $C^\alpha_\ell(s,t)$, respectively, 
we obtain a homogeneous differential operator
$\mathcal{C^\alpha_\ell}:=
C^\alpha_\ell\left(-\frac{\partial^2}{\partial x^2}, \frac{\partial}{\partial y} \right)$
of order $\ell$ on 
$\mathbb{R}^2$.
Here are the first four operators:
\begin{enumerate}
\item[] $\mathcal{C}^\alpha_0=\text{id}$,
\item[] $\mathcal{C}^\alpha_1 = 2 \alpha \frac{\partial}{\partial y}$,
\item[] $\mathcal{C}^\alpha_2 = 
\alpha \left(-\frac{\partial^2}{\partial x^2}+ (\alpha + 1) \frac{\partial^2}{\partial y^2}\right)$,
\item[] $\mathcal{C}^\alpha_3 = \frac{2}{3} \alpha (\alpha +1)
\left( 3 \frac{\partial^3}{\partial x ^2 \partial y} + 2(\alpha + 2) \frac{\partial^3}{\partial y^3}\right)$.
\end{enumerate}

\begin{thmalph}\label{thm:A}
Suppose that $a:=\nu - \lambda$ is a non-negative integer.
For $a > 0$, we 
define the following pair of homogeneous differential operators
of order $a$ on $\mathbb{R}^2$ by
\begin{align*}
\mathcal{D}_1 &: = a (2\lambda+a -1) \frac{\partial}{\partial x}\circ 
\mathcal{C}^{\lambda+\frac{1}{2}}_{a-1}\\
\mathcal{D}_2 &:=\left(2\lambda^2 + 2(a-1)\lambda + a(a-1) \right)
\frac{\partial}{\partial y} \circ\mathcal{C}^{\lambda+\frac{1}{2}}_{a-1}\\
&\hskip 1.2in 
+(\lambda-1)(2\lambda+1) 
\left(\frac{\partial^2}{\partial x^2} + \frac{\partial^2}{\partial y^2}\right)
\circ\mathcal{C}^{\lambda+\frac{3}{2}}_{a-2}.
\end{align*}
For $a=0$, we set
\begin{equation*}
\mathcal{D}_1 := \emph{id}, \quad \mathcal{D}_2:=0.
\end{equation*}
Then $\mathcal{D}: = \emph{Rest}_{y = 0}
\circ(\mathcal{D}_1, \mathcal{D}_2)$ 
and $\mathcal{D}^\vee := \Rest{y=0} \circ 
(\mathcal{D}_2, -\mathcal{D}_1)$ satisfy
the functional identity \eqref{eqn:A}.
Moreover, when $2\lambda \notin \{0,-1,-2, \cdots\}$,
there exists 
a non-trivial solution to \eqref{eqn:A} 
only if $\nu - \lambda$ is a non-negative integer.
and any differential operator satisfying
\eqref{eqn:A} is a linear combination of 
$\mathcal{D}$ and $\mathcal{D}^\vee$.
\end{thmalph}

\textbf{Notation:} $\mathbb{N} := \{0, 1, 2, \ldots\}$\\
\phantom{a} \hskip 0.96in $\mathbb{N}_+ := \{1, 2, \ldots\}$

\section{Three equivalent formulations}\label{sec:2}

Question \ref{quest:A} arises from various disciplines
of mathematics. In this section we describe it in three equivalent ways.

\subsection{Covariance of $SL(2,\mathbb{R})$ for vector-valued functions}

For  $\lambda \in \mathbb{C}$, we define a group homomorphism
\begin{equation}\label{eqn:Fhl}
\psi_\lambda: \mathbb{C}^\times \to GL(2, \mathbb{R}),
\quad z=re^{i\theta} \mapsto r^\lambda
\begin{pmatrix} \cos \theta & \sin \theta \\ -\sin \theta & \cos \theta \end{pmatrix}.
\end{equation}
For a $\mathbb{C}^2$-valued function $\vec{F}$ on 
$\mathbb{C} \simeq \mathbb{R}^2$, we set 
\begin{equation*}
\vec{F}^h_\lambda(z) : = 
\psi_\lambda\left( (cz+d)^{-2} \right) \vec{F}
\bigg(\frac{az + b}{cz+d}\bigg)
\end{equation*}
for 
$\lambda \in \mathbb{C}$,
$h^{-1} = \begin{pmatrix} a & b \\ c & d \end{pmatrix} \in SL(2,\mathbb{R})$,
and $z \in \mathbb{C}$ such that $cz+d \neq 0$.

\begin{pquestion}\label{quest:A1}
(1) 
Determine complex parameters $\lambda, \nu \in \mathbb{C}$ 
for which there exist differential operators 
$\mathcal{D}_1$ and $\mathcal{D}_2$ on $\mathbb{R}^2$
with the following property:
$\mathcal{D}= \Rest{y=0} \circ 
(\mathcal{D}_1, \mathcal{D}_2)$ satisfies
\begin{equation}\label{eqn:A2}
(\mathcal{D}\vec{F}^h_\lambda)(x)= |cx+d|^{-2\nu}
(\mathcal{D}\vec{F})\bigg(\frac{ax+b}{cx+d}\bigg)
\end{equation}
for all $\vec{F} \in 
C^\infty(\mathbb{C}) \oplus C^\infty(\mathbb{C}),
h^{-1} = \begin{pmatrix} a & b \\ c & d \end{pmatrix} \in SL(2, \mathbb{R})$,
and $x \in \mathbb{R} \setminus \{-\frac{d}{c}\}$.\\

\noindent
(2) Find an explicit formula of such $\mathcal{D} \equiv 
\mathcal{D}_{\lambda, \nu}$.

\end{pquestion}

The equivalence between Questions \ref{quest:A} and \ref{quest:A1}
follows from the following three observations:
\begin{itemize}
\item The functional identity \eqref{eqn:A2}
for $h = \begin{pmatrix}1 & t \\ 0 & 1 \end{pmatrix} $ 
$(t \in \mathbb{R})$ implies that 
$\mathcal{D}=\Rest{y=0} \circ 
(\mathcal{D}_1, \mathcal{D}_2)$ is a translation invariant operator.
Therefore, we can take $\mathcal{D}_1$ and $\mathcal{D}_2$ to have constant coefficients.

\item $\vec{F}^\vee_\lambda = \vec{F}^w_\lambda$.

\item The group $SL(2,\mathbb{R})$ is generated by 
$w$ and  
$\left\{\begin{pmatrix} 1 & t \\ 0 & 1 \end{pmatrix}: t \in \mathbb{R}\right\}$.

\end{itemize}

\subsection{Conformally covariant differential operators}\label{subsec:conf}

Let $X$ be a smooth manifold equipped with a Riemannian
metric $g$. Suppose that a group $G$ acts on $X$ by the map
$G \times X \to X$, $(h, x) \mapsto h \cdot x$. 
This action is called \emph{conformal} if there is a positive-valued
smooth function (\emph{conformal factor}) $\Omega$ on $G \times X$
such that
\begin{equation*}
h^*(g_{h\cdot x}) = \Omega(h,x)^2 g_x 
\quad \text{for any $h \in G$ and $x \in X$.}
\end{equation*}

Given $\lambda \in \mathbb{C}$, we define a $G$-equivariant
line bundle $\mathcal{L}_\lambda\equiv \cL{\lambda}$ 
over $X$ by letting $G$
act on the direct product $X \times \mathbb{C}$ by 
$(x, u) \mapsto (h\cdot x, \Omega(h, x)^{-\lambda}u)$
for $h \in G$. Then we have a natural action of $G$ on the 
vector space $\mathcal{E}_\lambda(X) := C^\infty(X, \mathcal{L}_\lambda)$
consisting of smooth sections for $\mathcal{L}_\lambda$.
Since $\mathcal{L}_\lambda \to X$ is topologically a trivial bundle,
we may identify $\mathcal{E}_\lambda(X)$ with $C^\infty(X)$, and 
corresponding $G$-action on $C^\infty(X)$ is given as the multiplier
representation $\varpi_\lambda \equiv \varpi_\lambda^X$:
\begin{equation*}
\left(\varpi_\lambda(h)f\right)(x)
=\Omega(h^{-1},x)^\lambda f(h^{-1}\cdot x)
\quad \text{for $h \in G$ and $f \in C^\infty(X)$.}
\end{equation*}
See \cite{xkors} for the basic properties of the representation
$(\varpi_\lambda, C^\infty(X))$. 

\begin{example} \label{ex:SL2}
We endow $\mathbb{P}^1\mathbb{C} \simeq \mathbb{C}\cup \{\infty\}$ 
with a Riemannian metric $g$ 
via the stereographic projection of the unit sphere $S^2$:
\begin{equation*}
\mathbb{R}^3 \supset S^2 \stackrel{\sim}{\to} \mathbb{C}\cup \{\infty\},
\quad
(p,q,r) \mapsto \frac{p+\sqrt{-1}q}{1+r}.
\end{equation*}
Then $g(u,v) = \frac{4}{(1+|z|^2)^2}(u,v)_{\mathbb{R}^2}$
for $u, v \in T_z\mathbb{C} \simeq \mathbb{R}^2$,
and the M\"obius transformation, defined by
\begin{equation*}
\mathbb{P}^1\mathbb{C} \to \mathbb{P}^1\mathbb{C}, \quad 
z \mapsto g\cdot z =\frac{az+b}{cz+d} 
\quad
\text{for $g = \begin{pmatrix} a & b \\ c&d\end{pmatrix} \in 
SL(2,\mathbb{C})$},
\end{equation*}
is conformal with conformal factor
\begin{equation}\label{eqn:czd}
\Omega(g,z) = |cz+d|^{-2}.
\end{equation}
Therefore,
\begin{equation*}
\left( \varpi_\lambda(h)f\right)(z) 
= |cz+d|^{-2\lambda}
f\bigg(\frac{az+b}{cz+d}\bigg)
\quad
\text{for $h^{-1} = \begin{pmatrix}a & b \\ c & d \end{pmatrix}$.}
\end{equation*}
This is a (non-unitary) spherical principal series representation
$\Ind_{B_\mathbb{C}}^{G_\mathbb{C}}(1\otimes \lambda \alpha \otimes 1)$
of $G_\mathbb{C}= SL(2,\mathbb{C})$, where $\alpha$ is the unique
positive restricted root which defines a Borel subgroup $B_\mathbb{C}$.
\end{example}

Let $\wedge^i T^*X$ be
the $i$-th exterior power of the cotangent
bundle $T^*X$
for $0\leq i \leq n$, where $n$ is the dimension of $X$.
Then sections $\omega$ for $\wedge^i T^*X$ are 
$i$-th differential forms on $X$, and $G$ acts on 
$\mathcal{E}^i(X) = C^\infty(X, \wedge^i T^*X)$
as the pull-back of differential forms:
\begin{equation*}
\varpi(h)\omega = (h^{-1})^*\omega
\quad
\text{for $\omega \in \mathcal{E}^i(X)$.}
\end{equation*}
More generally, the tensor bundle 
$\mathcal{L}_\lambda \otimes 
\wedge^i T^*X$ is also a $G$-equivariant vector bundle
over $X$, and we denote by $\varpi_{\lambda, i}^X$
the regular representation
of $G$ on the space of sections
\begin{equation*}
\mathcal{E}^i_\lambda(X):=
C^\infty(X, \mathcal{L}_\lambda \otimes 
\wedge^i T^*X).
\end{equation*}

By definition $\mathcal{E}^0_\lambda(X) = \mathcal{E}_\lambda(X)$.
In our normalization we have a natural $G$-isomorphism
\begin{equation*}
\mathcal{E}^0_n(X) \simeq \mathcal{E}^n_0(X),
\end{equation*}
if $X$ admits a $G$-invariant orientation.

Denote by $\text{Conf}(X)$ the full group of conformal
transformations of the Riemannian manifold 
$(X, g)$. Given a submanifold $Y$ of $X$, we define
a subgroup by 
\begin{equation*}
\text{Conf}(X; Y) := 
\{\varphi \in \text{Conf}(X) : \varphi(Y) = Y\}.
\end{equation*}
Then the induced action of 
$\text{Conf}(X;Y)$ on 
the Riemannian manifold $(Y, g|_Y)$
is again conformal.
We then consider the following problem.

\begin{problem}\label{prob:form}
(1) Given $0\leq i \leq \dim X$ and 
$0 \leq j \leq \dim Y$,
classify $(\lambda, \nu) \in \mathbb{C}^2$ such that 
there exists a non-zero local/non-local operator
\begin{equation*}
T: \mathcal{E}^i_\lambda(X) \to \mathcal{E}^j_\nu(Y)
\end{equation*}
satisfying
\begin{equation*}
\varpi_{\nu,j}^Y(h) \circ T 
= T \circ \varpi_{\lambda, i}^X(h)
\quad
\text{for all $h \in \text{Conf}(X;Y)$.}\\
\end{equation*}

\noindent
(2) Find explicit \text{formul\ae}  of the operators
$T \equiv T_{\lambda, \nu}^{i,j}$.
\end{problem}

The case $i = j = 0$ is a question that was raised 
in \cite[Problem 4.2]{Kob14a} as a geometric aspect of the 
branching problem for representations with respect to
the pair of groups $\text{Conf}(X) \supset \text{Conf}(X;Y)$.

As a special case, one may ask:

\begin{ppquestion}\label{quest:A2}
Solve Problem \ref{prob:form} for covariant differential operators
in the setting that $(X,Y) = (S^2, S^1)$ and $(i,j) = (1,0)$.
\end{ppquestion}

We note that, for $(X,Y) = (S^2, S^1)$,
there are natural homomorphisms
\begin{alignat*}{3}
   G_\mathbb{C}&:=SL(2,\mathbb{C})&& \rightarrow  &&\operatorname{Conf}(X) 
\\
    &\hphantom{mmmm}\cup &&     && \hphantom{mm}\cup
\\
   G_\mathbb{R}&:=SL(2,\mathbb{R})&&  \rightarrow &&
   \operatorname{Conf}(X;Y), 
\end{alignat*}
\noindent
and the images of $SL(2,\mathbb{C})$ and $SL(2,\mathbb{R})$
coincide with the identity component groups of 
$\operatorname{Conf}(X)\simeq O(3,1)$
and $\operatorname{Conf}(X;Y)$, respectively.
Question \ref{quest:A} is equivalent to Question \ref{quest:A2}
with $\operatorname{Conf}(X;Y)$ replaced by its identity component
$SO_0(2,1) \simeq SL(2,\mathbb{R})/\{\pm I\}$.
In fact, the differential operator
$\mathcal{D}=\Rest{y=0} \circ (\mathcal{D}_1, \mathcal{D}_2)$ 
in Question \ref{quest:A} gives a $G_\mathbb{R}$-equivariant
differential operator
\begin{equation*}
\mathcal{E}^1_{\lambda-1}(S^2) 
\to \mathcal{E}^0_\nu(S^1) \equiv \mathcal{E}_\nu(S^1)
\end{equation*}
in our normalization, which takes the form
\begin{equation*}
\mathcal{E}^1(\mathbb{R}^2) \to C^\infty(\mathbb{R}),
\quad fdx + gdy \mapsto 
(\mathcal{D}_1 f) (x, 0) + (\mathcal{D}_2g)(x,0)
\end{equation*}
in the flat coordinates via the stereographic projection.

\subsection{Branching laws of Verma modules}\label{subsec:2.3}

Let $\mathfrak{g} = \mathfrak{sl}(2, \mathbb{C})$, and $\mathfrak{b}$ 
a Borel subalgebra consisting of lower 
triangular matrices in $\mathfrak{g}$.
For $\lambda \in \mathbb{C}$, we define a character of $\mathfrak{b}$,
to be denoted by $\mathbb{C}_\lambda$, as
\begin{equation*}
\mathfrak{b} \to \mathbb{C}, 
\quad
\begin{pmatrix}
-x & 0 \\ y & x 
\end{pmatrix}
\mapsto \lambda x.
\end{equation*}
If $\lambda \in \mathbb{Z}$ then
$\mathbb{C}_\lambda$ is the differential of the holomorphic
character $\chi_{\lambda, \lambda}$ of 
the Borel subgroup $B_\mathbb{C}$, which will be defined in
\eqref{eqn:chi} in Section \ref{subsec:3.1}.

We consider a $\mathfrak{g}$-module, referred to as a Verma module,
defined by
\begin{equation*}
M(\lambda) := U(\mathfrak{g}) \otimes_{U(\mathfrak{b})}\mathbb{C}_\lambda.
\end{equation*}
Then $\mathbf{1}_\lambda := 1 \otimes 1 \in M(\lambda)$
is a highest weight vector with weight $\lambda \in \mathbb{C}$, and it generates
$M(\lambda)$ as a $\mathfrak{g}$-module. The $\mathfrak{g}$-module
$M(\lambda)$ is irreducible if and only if $\lambda \notin \mathbb{N}$.

We consider the following algebraic question:

\begin{pppquestion}\label{quest:Verma}
(1) Classify $(\mu, \lambda_1, \lambda_2) \in \mathbb{C}^3$ such that
\begin{equation*}
\operatorname{Hom}_\mathfrak{g}
\left(M(\mu), M(\lambda_1) \otimes M(\lambda_2)\right)
\neq \{0\}.
\end{equation*}
(2) Find an explicit expression of $\varphi(\mathbf{1}_\mu)$
in $M(\lambda_1) \otimes M(\lambda_2)$ for any
$\varphi \in \operatorname{Hom}_\mathfrak{g}
\left( M(\mu), M(\lambda_1) \otimes M(\lambda_2)\right)$.
\end{pppquestion}

An answer to Question \ref{quest:Verma} is given as follows:

\begin{proposition}\label{prop:Verma}
If $\lambda_1 + \lambda_2 \notin \mathbb{N}$ then 
the tensor product $M(\lambda_1) \otimes M(\lambda_2)$
decomposes into the direct sum of Verma modules as follows:
\begin{equation*}
M(\lambda_1) \otimes M(\lambda_2) \simeq
\bigoplus_{a=0}^\infty M(\lambda_1 + \lambda_2 -2a).
\end{equation*}
\end{proposition}

For the proof, consult \cite{K12} for instance.
In fact, in \cite{K12}, one finds the (abstract) branching laws of (parabolic)
Verma modules in the general setting of the restriction with 
respect to symmetric pairs. By the duality theorem
(\cite{KOSS}, \cite[Theorem 2.7]{KP})
between 
differential \emph{symmetry breaking operators}
(covariant differential operators to submanifolds)
and (discretely decomposable) branching laws of 
Verma modules, we have the following one-to-one correspondence
\begin{align}
&\{\text{The differential operators 
$\mathcal{D}$ yielding the functional identity \eqref{eqn:A}}\} \nonumber\\
&\leftrightarrow
\operatorname{Hom}_\mathfrak{g}
\left(M(-2\nu), M(-\lambda-1)\otimes M(-\lambda+1) \right) \label{eqn:KP}\\
&\hskip 1.6in \oplus 
\operatorname{Hom}_\mathfrak{g}
\left(M(-2\nu), M(-\lambda+1)\otimes M(-\lambda-1) \right), \nonumber
\end{align}
because 
$T_o (G_\mathbb{C}/ B_\mathbb{C}) \otimes \mathbb{C}
\simeq \mathbb{C}_{-2} \boxtimes \mathbb{C} + \mathbb{C} \boxtimes \mathbb{C}_{-2}$
as $\mathfrak{b} \otimes \mathbb{C} \simeq \mathfrak{b} \oplus \mathfrak{b}$-modules.
Combining this with Proposition \ref{prop:Verma}, we obtain

\begin{proposition}\label{prop:Verm}
If $2\lambda \notin - \mathbb{N}$ then a non-zero differential operator
$\mathcal{D}$ satisfying \eqref{eqn:A} exists if and only if 
$\nu - \lambda \in \mathbb{N}$,
and the set of such differential operators 
forms a two-dimensional vector space.
\end{proposition}

Owing to Proposition \ref{prop:Ddual}, we get the two-dimensional
solution space as the linear span of $\mathcal{D}$ and $\mathcal{D}^\vee$,
once we find a generic solution $\mathcal{D}$.

\section{Rankin--Cohen brackets}\label{sec:3}

As a preparation for the proof of Theorem \ref{quest:A},
we briefly review the Rankin--Cohen brackets, which
originated in number theory \cite{C75, EZ85, Ra56}.

\subsection{Homogeneous line bundles over $\mathbb{P}^1\mathbb{C}$}
\label{subsec:3.1}

First, we shall fix a normalization of three homogeneous line bundles over 
$X=\mathbb{P}^1\mathbb{C}$, namely, 
$\cL{\lambda}$ (Section \ref{sec:2}),
$\hL{\lambda}$, and $\mathcal{L}_{n,\lambda}$.

We define a Borel subgroup of $G_\mathbb{C} = SL(2, \mathbb{C})$ by
\begin{equation*}
B_\mathbb{C} := \bigg\{ \begin{pmatrix} a & 0 \\ c & \frac{1}{a} \end{pmatrix} :
a \in \mathbb{C}^\times, c \in \mathbb{C} \bigg\},
\end{equation*}
and identify $G_\mathbb{C} / B_\mathbb{C}$ with $X = \mathbb{P}^1 \mathbb{C}$
by $h B_\mathbb{C} \mapsto h \cdot 0$.

Given $n \in \mathbb{Z}$ and $\lambda \in \mathbb{C}$, 
we define a one-dimensional
representation of $B_\mathbb{C}$ by 
\begin{equation}\label{eqn:chi}
\chi_{n,\lambda} : B_\mathbb{C} \to \mathbb{C}^\times,
\quad
\begin{pmatrix}
\frac{1}{re^{i\theta}} & 0 \\
c & re^{i\theta}
\end{pmatrix}
\mapsto
e^{in\theta}r^\lambda,
\end{equation}
and a $G_\mathbb{C}$-equivariant line bundle 
$\mathcal{L}_{n,\lambda} = G_\mathbb{C} \times_{B_\mathbb{C}} \chi_{n,\lambda}$
as the set of equivalence classes of $G_\mathbb{C} \times \mathbb{C}$ given by
\begin{equation*}
(g, u) \sim (gb^{-1}, \chi_{n,\lambda}(b) u) 
\quad
\text{for some $b \in B_\mathbb{C}$.}
\end{equation*}
The conformal line bundle $\cL{\lambda}$ defined
in Section \ref{subsec:conf} amounts to 
$\mathcal{L}_{0, 2\lambda}$ by the formula \eqref{eqn:czd}.

On the other hand, if $\lambda = n \in \mathbb{Z}$ then $\chi_{\lambda,\lambda}$
is a holomorphic character of $B_\mathbb{C}$, and 
consequently, $\mathcal{L}_{\lambda, \lambda} \to X$ becomes
a holomorphic line bundle, which we denote by $\hL{\lambda}$.
The complexified cotangent bundle 
$(T^*X)\otimes \mathbb{C}$ splits into a Whitney sum of the 
holomorphic and anti-holomorphic cotangent bundle
$(T^*X)^{1,0} \oplus (T^* X)^{0,1}$,
which amounts to $\mathcal{L}_{2,2} \oplus \mathcal{L}_{-2,2}$.
In summary, we have:

\begin{lemma}\label{lem:Lpara}
We have the following isomorphisms of $G_\mathbb{C}$-equivariant
line bundles over $X \simeq \mathbb{P}^1\mathbb{C}$.
\begin{align*}
\hL{\lambda} &\simeq \mathcal{L}_{\lambda, \lambda} 
\quad \hskip 0.07in\text{for $\lambda \in \mathbb{Z}$,}\\
\cL{\lambda} &\simeq \mathcal{L}_{0, 2\lambda}
\quad \text{for $\lambda \in \mathbb{C}$,}\\
(T^*X)^{1, 0} &\simeq \mathcal{L}_{2,2},\\
(T^*X)^{0,1}  &\simeq \mathcal{L}_{-2,2}.
\end{align*}
\end{lemma}

The line bundle $\mathcal{L}_{n,\lambda} \to X$ 
is $G_\mathbb{C}$-equivariant; thus, there is
the regular representation $\pi_{n,\lambda}$
of $G_\mathbb{C}$
on $C^\infty(X, \mathcal{L}_{n,\lambda})$.
This is called the (unnormalized, non-unitary) principal 
series representation of $G_\mathbb{C}$.
The restriction to the open Bruhat cell 
$\mathbb{C} \hookrightarrow X=\mathbb{C} \cup \{\infty\}$
yields an injection
$C^\infty(X, \mathcal{L}_{n,\lambda}) \hookrightarrow C^\infty(\mathbb{C})$,
on which $\pi_{n,\lambda}$ is given as a multiplier representation:
\begin{equation*}
\left( \pi_{n,\lambda}(h) F\right) (z) = 
\bigg(\frac{cz+d}{|cz+d|} \bigg)^{-n} 
|cz+d|^{-\lambda}
F\bigg(\frac{az+b}{cz+d}\bigg)
\quad
\text{for $h^{-1} = \begin{pmatrix} a & b \\ c & d \end{pmatrix}$}.
\end{equation*}
Comparing this with the conformal construction
of the representation $\varpi_\lambda$ in 
Example \ref{ex:SL2}, we have $\varpi_\lambda \simeq \pi_{0,2\lambda}$.

Similarly to the smooth line bundle $\mathcal{L}_{n,\lambda}$,
we consider holomorphic sections for the holomorphic line bundle
$\hL{\lambda}$. For this, let $D$ be a domain of $\mathbb{C}$ and 
$G$ a subgroup of $G_\mathbb{C}$, which leaves $D$ invariant.
Then we can define a representation, to be denoted by 
$\hpi{\lambda}$, of $G$ on the space 
$\mathcal{O}(D) \equiv \mathcal{O}(D, \hL{\lambda})$
of holomorphic sections, which is identified with a multiplier representation
\begin{equation*}
\left(\hpi{\lambda}(h) F \right)(z)
=(cz+d)^{-\lambda}
F\bigg(\frac{az+b}{cz+d}\bigg)
\quad
\text{for $F \in \mathcal{O}(D)$}.
\end{equation*}

\begin{example}
(1) $D = \{ z \in \mathbb{C} : |z| < 1\}$, $G = SU(1,1)$.\\
\noindent 
\phantom{} 
\hskip 1.08in 
(2) $D = \{ z \in \mathbb{C} : \text{Im} z > 0\}$, $G = SL(2, \mathbb{R})$.\\
\noindent
(For the application below we shall use the unit disc model.)
\end{example}

\subsection{Rankin--Cohen bidifferential operator}\label{subsec:RC}

Let $D$ be a domain in $\mathbb{C}$. For $a \in \mathbb{N}$
and $\lambda_1, \lambda_2 \in \mathbb{C}$, 
the bidifferential operator 
$\adRC{\lambda_1}{\lambda_2}: 
\mathcal{O}(D) \otimes \mathcal{O}(D) \to \mathcal{O}(D)$, 
referred to as the \emph{Rankin--Cohen bracket} \cite{C75, Ra56},
is defined by
\begin{equation*}
\adRC{\lambda_1}{\lambda_2}(f_1 \otimes f_2)(z)
:=
\sum_{\ell = 0}^a (-1)^\ell 
\binom{\lambda_1 + a - 1}{\ell}
\binom{\lambda_2 + a - 1}{a-\ell}
\frac{\partial^{a-\ell}f_1}{\partial z^{a-\ell}}(z)
\frac{\partial^{\ell}f_2}{\partial z^{\ell}}(z).
\end{equation*}

In the theory of automorphic forms, 
$\adRC{\lambda_1}{\lambda_2}$
yields a new holomorphic modular form
of weight $\lambda_1 + \lambda_2 + 2a$ out of two
holomorphic modular forms $f_1$ and $f_2$ of weights 
$\lambda_1$ and $\lambda_2$, respectively.

From the viewpoint of representation theory, 
$\adRC{\lambda_1}{\lambda_2}$
is an intertwining operator:
\begin{equation}\label{eqn:RCcov}
\hpi{\lambda_1 + \lambda_2 + 2a}(h)
\circ \adRC{\lambda_1}{\lambda_2}
=
\adRC{\lambda_1}{\lambda_2} \circ
\left( \hpi{\lambda_1}(h) \otimes \hpi{\lambda_2}(h)\right)
\end{equation}
for all $h \in G$.

The coefficients of the Rankin--Cohen brackets look
somewhat complicated.
Eicheler--Zagier \cite[Chapter 3]{EZ85}
found that they are related to those of a 
classical orthogonal polynomial.
A short proof for this fact is given by the 
\emph{F-method} in \cite{KP}.

To see the relation, we define a polynomial 
$\aRC{\lambda_1}{\lambda_2}(x,y)$
of two variables
$x$ and $y$ by
\begin{equation}\label{eqn:RCxy}
\aRC{\lambda_1}{\lambda_2}(x,y):=
\sum_{\ell=0}^a (-1)^\ell
\binom{\lambda_1 + a -1}{\ell}
\binom{\lambda_2 + a -1}{a-\ell}
x^{a-\ell}y^\ell,
\end{equation}
so that the Rankin--Cohen bidifferential operator
$\mathcal{RC}^a_{\lambda_1, \lambda_2}$ 
is given by
\begin{equation*}
\mathcal{RC}^a_{\lambda_1, \lambda_2}
=\Rest{z_1 = z_2 = z}\circ 
\aRC{\lambda_1}{\lambda_2}\left(
\frac{\partial}{\partial z_1}, \frac{\partial}{\partial z_2}
\right).
\end{equation*}
The polynomial $\aRC{\lambda_1}{\lambda_2}(x,y)$ is 
of  homogeneous degree $a$. Clearly we have:

\begin{lemma}\label{lem:RCdual}
$\aRC{\lambda_1}{\lambda_2} (x,y) = (-1)^a\aRC{\lambda_2}{\lambda_1}(y,x)$.
\end{lemma}

Second we recall that
the Jacobi polynomial $P^{\alpha,\beta}_\ell(t)$
is a polynomial of one variable $t$ of degree $\ell$
given by
\begin{equation*}
P^{\alpha, \beta}_\ell(t) = 
\frac{\Gamma(\alpha + \ell + 1)}{\Gamma(\alpha +\beta + \ell+ 1)}
\sum_{m=0}^\ell 
\frac{\Gamma(\alpha + \beta + \ell + m +1)}
{(\ell-m)! m! \Gamma(\alpha + m + 1)}
\bigg(\frac{t-1}{2}\bigg)^m.
\end{equation*}

We inflate it to a homogeneous polynomial of two
variables $x$ and $y$ of degree $\ell$ by
\begin{equation*}
P^{\alpha, \beta}_{\ell}(x,y):=
y^\ell P^{\alpha, \beta}_\ell \bigg(2 \frac{x}{y}+1 \bigg).
\end{equation*}
For instance, $P^{\alpha, \beta}_0(x,y)=1$ and
$P^{\alpha,\beta}_1(x,y) = (2+\alpha + \beta)x+(\alpha+1)y$.
It turns out that 
\begin{equation*}
\aRC{\lambda_1}{\lambda_2}(x,y)=(-1)^a
P^{\lambda_1-1, -\lambda_1 - \lambda_2 - 2a + 1}_a(x,y).
\end{equation*}
In particular, the following holds.

\begin{lemma}\label{lem:RCJacobi}
We have
\begin{equation*}
\adRC{\lambda_1}{\lambda_2} = (-1)^a 
\Rest{z_1 = z_2=z} \circ
P^{\lambda_1-1, -\lambda_1 - \lambda_2 - 2a + 1}_a
\left(\frac{\partial}{\partial z_1}, \frac{\partial}{\partial z_2}\right).
\end{equation*}
\end{lemma}

\section{Holomorphic trick}\label{sec:htrick}

In this section we give a proof for Theorem \ref{quest:A}
by using the results of the previous sections

\subsection{Restriction to a totally real submanifold}\label{subsec:real}

Consider a totally real embedding of 
$X = \mathbb{P}^1\mathbb{C} \simeq \mathbb{C} \cup \{\infty\}$
defined by
\begin{equation}\label{eqn:PCzz}
\iota: \mathbb{P}^1\mathbb{C} 
\to \mathbb{P}^1\mathbb{C} \times \mathbb{P}^1\mathbb{C},
\quad
z\mapsto (z, \bar{z}).
\end{equation}
The map $\iota$ respects the action of $G_\mathbb{C}$ via the
following group homomorphism (we regard $G_\mathbb{C}$ as a real group), 
denoted by the same letter,
\begin{equation*}
\iota: G_\mathbb{C} \to G_\mathbb{C} \times G_\mathbb{C},
\quad
g \mapsto (g, \bar{g}).
\end{equation*}
This is because
$G_\mathbb{C}/B_\mathbb{C} \simeq \mathbb{P}^1\mathbb{C}$ and 
because the Borel subgroup $B_\mathbb{C}$ is stable by the complex
conjugation $g \mapsto \bar{g}$. Then the following lemma is immediate
from Lemma \ref{lem:Lpara}.

\begin{lemma} We have
an isomorphism of $G_\mathbb{C}$-equivariant line bundles:
\begin{equation*}
\iota^*\left(\hL{\lambda_1} \boxtimes \hL{\lambda_2}\right)
\simeq \mathcal{L}_{\lambda_1-\lambda_2, \lambda_1 + \lambda_2}.
\end{equation*}
In particular,
\begin{align}
&\iota^*(\hL{\lambda+1} \boxtimes 
\hL{\lambda-1} ) 
\simeq \cL{\lambda-1} \otimes (T^*X)^{1,0}, \label{eqn:pb1}\\
&\iota^*(\hL{\lambda-1} \boxtimes 
\hL{\lambda+1} ) 
\simeq \cL{\lambda-1} \otimes (T^*X)^{0,1} \label{eqn:pb2}.
\end{align}
\end{lemma}

\begin{proposition}\label{prop:htrick}
The isomorphisms \eqref{eqn:pb1} and \eqref{eqn:pb2} induce
injective $G_\mathbb{C}$-equivariant 
homomorphisms between equivariant sheaves:
\begin{align*}
&(\iota^*)^{1,0}: \mathcal{O}(\hL{\lambda+1}) 
\otimes \mathcal{O}(\hL{\lambda-1}) 
\to \mathcal{E}^{1,0}_{\lambda-1},
\quad f_1(z_1) \otimes f_2(z_2) \mapsto 
f_1(z) f_2(\bar{z}) dz,\\
&(\iota^*)^{0,1}: \mathcal{O}(\hL{\lambda-1}) 
\otimes \mathcal{O}(\hL{\lambda+1}) 
\to \mathcal{E}^{0,1}_{\lambda-1},
\quad f_1(z_1) \otimes f_2(z_2) \mapsto 
f_1(z) f_2(\bar{z}) d\bar{z},
\end{align*}
that is, $(\iota^*)^{1,0}$ and $(\iota^*)^{0,1}$ are injective
on every open set $D$ in $\mathbb{P}^1\mathbb{C}$, and
\begin{align*}
(\iota^*)^{1,0} \circ
\left( \hpi{\lambda+1}(g) \otimes \hpi{\lambda-1} (\bar g) \right)
&= \varpi^1_{\lambda-1}(g) \circ (\iota^*)^{1,0} \\
(\iota^*)^{0,1} \circ
\left( \hpi{\lambda-1}(g) \otimes \hpi{\lambda+1} (\bar g) \right)
&= \varpi^1_{\lambda-1}(g) \circ (\iota^*)^{0,1}
\end{align*}
hold for any $g$ whenever they make sense.
\end{proposition}

\begin{proof}
The injectivity follows from 
the identity theorem of holomorphic functions because
$\iota: \mathbb{P}^1\mathbb{C} \to 
\mathbb{P}^1\mathbb{C} \times \mathbb{P}^1\mathbb{C}$
is a totally real embedding. The covariance property is derived
from \eqref{eqn:pb1} and \eqref{eqn:pb2}.
\end{proof}

Fix $\lambda \in \mathbb{Z}$ and $a \in \mathbb{N}$,
and set $\nu = \lambda + a$. 
We want to relate the Rankin--Cohen brackets
$\adRC{\lambda \pm 1}{\lambda \mp 1}$ to our 
differential operator $\mathcal{D}$ (see Question \ref{quest:A})
in the sense that
both of the following diagrams commute:

\begin{alignat*}{3}
\mathcal{O}(\hL{\lambda+1}) \otimes\mathcal{O}(\hL{\lambda-1})
&\stackrel{(\iota^*)^{1,0}}{\lhook\joinrel\relbar\joinrel\relbar\joinrel\relbar\joinrel\relbar}
&
\!\!\joinrel\rightarrow 
\mathcal{E}^{1,0}_{\lambda-1}(\mathbb{C})& 
\subset \mathcal{E}^{1}_{\lambda-1}(\mathbb{R}^2)  &
&\simeq C^\infty(\mathbb{R}^2) \oplus C^\infty(\mathbb{R}^2)\\
{\scriptstyle \adRC{\lambda+1}{\lambda-1}} {\big \downarrow} \hskip 48pt
& &&     &&\hskip 60pt \big \downarrow {\scriptstyle\mathcal{D}}\\
\mathcal{O}(\hL{2\lambda+2a}) \phantom{mm} 
& 
\stackrel{\iota^*}{\lhook\joinrel\relbar\joinrel\relbar\joinrel\relbar\joinrel\relbar}&
\joinrel\relbar\joinrel\relbar\joinrel\relbar
\joinrel\relbar\joinrel\relbar\joinrel\relbar\joinrel\relbar&
\!\!\joinrel\rightarrow \hskip 2pt
\mathcal{E}_\nu(\mathbb{R}) &&
\simeq  \hskip 35pt C^\infty(\mathbb{R}), 
\end{alignat*}
and
\begin{alignat*}{3}
\mathcal{O}(\hL{\lambda-1}) \otimes\mathcal{O}(\hL{\lambda+1})
&\stackrel{(\iota^*)^{0,1}}{\lhook\joinrel\relbar\joinrel\relbar\joinrel\relbar\joinrel\relbar}
&
\!\!\joinrel\rightarrow 
\mathcal{E}^{0,1}_{\lambda-1}(\mathbb{C})& 
\subset \mathcal{E}^{1}_{\lambda-1}(\mathbb{R}^2)  &
&\simeq C^\infty(\mathbb{R}^2) \oplus C^\infty(\mathbb{R}^2)\\
{\scriptstyle (-1)^a\adRC{\lambda-1}{\lambda+1}} {\big \downarrow} \hskip 48pt
& &&     &&\hskip 60pt \big \downarrow {\scriptstyle\mathcal{D}}\\
\mathcal{O}(\hL{2\lambda+2a}) \phantom{mm} 
& 
\stackrel{\iota^*}{\lhook\joinrel\relbar\joinrel\relbar\joinrel\relbar\joinrel\relbar}&
\joinrel\relbar\joinrel\relbar\joinrel\relbar
\joinrel\relbar\joinrel\relbar\joinrel\relbar\joinrel\relbar&
\!\!\joinrel\rightarrow \hskip 2pt
\mathcal{E}_\nu(\mathbb{R}) &&
\simeq  \hskip 35pt C^\infty(\mathbb{R}).
\end{alignat*}
Here we have used the following identification:
\begin{equation*}
\mathcal{E}^1_{\lambda-1}(\mathbb{R}^2) \simeq
C^\infty(\mathbb{R}^2) \oplus C^\infty(\mathbb{R}^2),
\quad
fdx+gdy \mapsto (f, g).
\end{equation*}

We define homogeneous polynomials
$D_1$, $D_2$ with real coefficients so that
\begin{equation*}
D_1(x,y) + \sqrt{-1} D_2(x,y) = 
2^{-a}\aRC{\lambda+1}{\lambda-1}
(x-\sqrt{-1}y, x+\sqrt{-1}y),
\end{equation*}
where $\aRC{\lambda_1}{\lambda_2}(x,y)$ is a 
polynomial defined in \eqref{eqn:RCxy}.
We set
\begin{align}\label{eqn:DRC}
\mathcal{D}_1:=D_1\left(\frac{\partial}{\partial x},\frac{\partial}{\partial y}  \right),
\quad
\mathcal{D}_2:=D_2\left(\frac{\partial}{\partial x}, \frac{\partial}{\partial y} \right),
\quad
\mathcal{D}:=\Rest{y=0} \circ \left(\mathcal{D}_1, \mathcal{D}_2 \right).
\end{align}

\begin{lemma}\label{lem:RCreal}
For any holomorphic functions $f_1$ and $f_2$,
\begin{align*}
&\mathcal{D}\left( (\iota^*)^{1,0}(f_1 \otimes f_2) \right)
= \iota^* \adRC{\lambda+1}{\lambda-1}(f_1 \otimes f_2),\\
&\mathcal{D}\left( (\iota^*)^{0,1}(f_1 \otimes f_2) \right)
= (-1)^a\iota^* \adRC{\lambda-1}{\lambda+1}(f_1 \otimes f_2).
\end{align*}
\end{lemma}

\begin{proof}
Let $\omega:=(\iota^*)^{1,0}(f_1 \otimes f_2) = f_1(z)f_2(\bar{z})dz$.
If we write $\omega = fdx + gdy$ then $f(z) = f_1 (z) f_2(\bar{z})$
and $g = \sqrt{-1}f$. Therefore,
\begin{align*}
\mathcal{D}\omega &= \Rest{y=0} \circ (\mathcal{D}_1, \mathcal{D}_2)
\begin{pmatrix} f \\ g \end{pmatrix},\\
(\mathcal{D}_1,\mathcal{D}_2) \begin{pmatrix} f \\ g \end{pmatrix}
&=\left(\mathcal{D}_1 + \sqrt{-1}\mathcal{D}_2) (f_1(z) f_2(\bar{z})\right).
\end{align*}
If we write 
$\aRC{\lambda+1}{\lambda-1}(x,y) = \sum_{\ell=0}^a r_\ell x^{a-\ell}y^\ell$
then 
\begin{align*}
(\mathcal{D}_1 + \sqrt{-1} \mathcal{D}_2)\left(f_1(z)f_2(\bar{z})\right)
&= \aRC{\lambda+1}{\lambda-1}
\bigg(\frac{\partial}{\partial z},
\frac{\partial}{\partial \bar{z}} \bigg)
(f_1(z)f_2(\bar{z}))\\
&= \sum_{\ell=0}^a r_\ell 
\frac{\partial^{a-\ell}f_1}{\partial z^{a-\ell}} (z)
\frac{\partial^{\ell}f_2}{\partial \bar{z}^{\ell}} (\bar{z}),
\end{align*}
because $f_1$ and $f_2$ are holomorphic. Taking
the restriction to $y=0$, we get
\begin{equation*}
\mathcal{D}(\omega)
= \sum_{\ell=0}^a r_\ell 
\frac{\partial^{a-\ell}f_1}{\partial x^{a-\ell}} (x)
\frac{\partial^{\ell}f_2}{\partial x^\ell} (x)
=\iota^* \adRC{\lambda+1}{\lambda-1}(f_1 \otimes f_2).
\end{equation*}
Hence we have proved the first identity. The second
identity follows from Lemma \ref{lem:RCdual}.
\end{proof}

\begin{remark}
If we multiply the bidifferential operator
$\adRC{\lambda+1}{\lambda-1}$ by
$\sqrt{-1}$ then
obviously \eqref{eqn:RCcov} holds, where the role of 
$(\mathcal{D}_1, \mathcal{D}_2)$ 
is changed into $(-\mathcal{D}_2, \mathcal{D}_1)$ because
\begin{equation*}
\sqrt{-1}
(\mathcal{D}_1 + \sqrt{-1}\mathcal{D}_2) 
= -\mathcal{D}_2 + \sqrt{-1}\mathcal{D}_1.
\end{equation*}
This explains Proposition \ref{prop:Ddual} from
the ``holomorphic trick."
\end{remark}

\subsection{Identities of Jacobi polynomials}\label{subsec:sp}

For $a \in \mathbb{N}_+$, we define the following three meromorphic
functions of $\lambda$ by

\begin{align*}
A_a(\lambda)&:= \frac{2 \lambda^2 + 2(a-1) \lambda + a(a-1)}
{a (2\lambda+a -1)},\\
B_a(\lambda)&:=\frac{(\lambda-1)(2\lambda+1)}
{a(2\lambda+a-1)}, \\
U_a(\lambda)&:=
\frac{2 \left(\lambda+ [\frac{a}{2}] \right)_{[\frac{a-1}{2}]}}
{\left(\lambda + \frac{1}{2} \right)_{[\frac{a-1}{2}]}},
\end{align*}
where 
$(\mu)_k:= \mu(\mu+1)\cdots (\mu+k-1) = 
\frac{\Gamma(\mu+k)}{\Gamma(\mu)}$
is the Pochhammer symbol.

\begin{proposition}\label{prop:Jacob}
For $a\in \mathbb{N}_+$, we have 
\begin{align*}
&(1-z)^aP^{\lambda, -2\lambda-2a+1}_a
\bigg(\frac{3+z}{1-z}\bigg)\\
&=(-1)^{a-1}U_a(\lambda) 
\bigg(
(1-A_a(\lambda)z)C^{\lambda+\frac{1}{2}}_{a-1}(z)+
B_a(\lambda)(1-z^2)C^{\lambda+\frac{3}{2}}_{a-2}(z) \bigg).
\end{align*}

Equivalently,\begin{align*}
&P^{\lambda, -2\lambda-2a+1}_a(x-\sqrt{-1}y, x+\sqrt{-1}y)\\
&=(\sqrt{-1})^{a-1}U_a(\lambda)\bigg(
xC^{\lambda+\frac{1}{2}}_{a-1}(-x^2, y) +
\sqrt{-1}\left( A_a(\lambda) y C^{\lambda+\frac{1}{2}}_{a-1}(-x^2, y)
+ B_a(\lambda) (x^2 + y^2) C^{\lambda+\frac{3}{2}}_{a-2}(-x^2, y)\right) \bigg).
\end{align*}
\end{proposition}

Proposition \ref{prop:Jacob} will be used
in the proof of Theorem \ref{thm:A} in the next subsection.
We want to note that we wondered if
the first equation of Proposition \ref{prop:Jacob}
was already known; however,  
we could not find the identity in the literature.

One might give an alternative proof of Proposition \ref{prop:Jacob}
by applying the F-method to a vector bundle case. 
We will discuss this approach in a subsequent paper.

\subsection{Proof of Theorem \ref{thm:A}}\label{subsec:proof}

The relations in Lemma \ref{lem:RCreal} 
and the covariance property \eqref{eqn:RCcov} 
of the Rankin--Cohen brackets
imply that the differential operator
$\mathcal{D}$ defined in \eqref{eqn:DRC}
satisfies the covariance relations \eqref{eqn:A2}
on the image 
\begin{equation*}
(\iota^*)^{1,0}
\big( \mathcal{O}(\hL{\lambda+1})\otimes 
\mathcal{O}(\hL{\lambda-1}) \big)
+ (\iota^*)^{0,1}
\big( \mathcal{O}(\hL{\lambda-1})\otimes 
\mathcal{O}(\hL{\lambda+1}) \big).
\end{equation*}

In order to prove \eqref{eqn:A2}, we need to show that
the image 
is dense in $C^\infty(\mathbb{R}^2) \oplus C^\infty(\mathbb{R}^2)$
topologized by uniform convergence on compact sets.
To see this we note that the image contains a linear span
of the following 	1-forms
\begin{equation*}
z^m\bar{z}^n dz, \quad
z^m\bar{z}^n d\bar{z}, \quad
(m, n \in \mathbb{N}).
\end{equation*}
Since a linear span of $(x+iy)^m(x-iy)^n$ $(m,n\in \mathbb{N})$
is dense in $C^\infty(\mathbb{R}^2)$ by the Stone--Weierstrass theorem,
we conclude that $\mathcal{D}$ satisfies \eqref{eqn:A2}.
An explicit formula for  the operators $(\mathcal{D}_1,\mathcal{D}_2)$ 
is derived from the Rankin--Cohen brackets by using
Lemma \ref{lem:RCJacobi} and 
Proposition \ref{prop:Jacob}
for $\lambda \in \mathbb{Z}$.
Then the covariance relations \eqref{eqn:invF}
are satisfied for all $\lambda \in \mathbb{C}$ because $\mathbb{Z}$
is Zariski dense in $\mathbb{C}$.

If $2\lambda \notin -\mathbb{N}$ then the dimension of solutions is 
two by Proposition \ref{prop:Verma} and the one-to-one 
correspondence \eqref{eqn:KP}.
Since $\mathcal{D}$ and $\mathcal{D}^\vee$ are linearly independent
for our solution $\mathcal{D}$, the linear span of $\mathcal{D}$ and 
$\mathcal{D}^\vee$ exhausts all the solutions by Proposition \ref{prop:Ddual}.
Hence Theorem \ref{thm:A} is proved.

\subsection{Scalar-valued case}\label{subsec:scalar}

So far we have discussed a family of vector-valued
differential operators that yield functional identities
satisfied by vector-valued functions. We close this 
article with some comments on the scalar-valued case.

Let $\lambda\in \mathbb{C}$. Given 
$f \in C^\infty(\mathbb{R}^2\setminus \{(0,0)\})
\simeq C^\infty(\mathbb{C} \setminus \{0\})$,
we define its \emph{twisted inversion} $f^\vee_\lambda$ by
\begin{equation*}
f^\vee_\lambda(r\cos \theta, r\sin \theta)
:=r^{-2\lambda}f\bigg(\frac{-\cos \theta}{r}, \frac{\sin \theta}{r}  \bigg)
\end{equation*}
as in \eqref{eqn:invF}, and more generally,
\begin{equation*}
f^h_\lambda(z):=
|cz+d|^{-2\lambda} f\left(\frac{az+b}{cz+d}\right) 
\quad
\text{for $h^{-1} = \begin{pmatrix} a & b\\ c & d \end{pmatrix} \in SL(2,\mathbb{C})$}
\end{equation*}
as in \eqref{eqn:Fhl}. 

For a differential operator $\mathcal{D}$ on $\mathbb{R}^2$,
we define a linear operator
$\tilde{\mathcal{D}} : C^\infty(\mathbb{R}^2) \to C^\infty(\mathbb{R})$
by
\begin{equation*}
\tilde{\mathcal{D}} := \Rest{y=0} \circ \mathcal{D}.
\end{equation*}

Fix $\lambda, \nu \in \mathbb{C}$.
As in Questions 
\ref{quest:A},
\ref{quest:A1},
\ref{quest:A2}, and
\ref{quest:Verma},
we may consider the following
equivalent questions:

\begin{question}\label{quest:B}
Find $\tilde{\mathcal{D}}$ with constant coefficients such that
\begin{equation*}
\left(\tilde{\mathcal{D}} f^\vee_\lambda\right)(x)
=|x|^{-2\nu} (\tilde{\mathcal{D}}f)\left(-\frac{1}{x}\right)
\quad
\text{for all $f \in C^\infty(\mathbb{R}^2)$ and $x \in \mathbb{R}^\times$.}
\end{equation*}
\end{question}

\begin{pquestion}\label{quest:B'}
Find $\tilde{\mathcal{D}}$ such that
\begin{equation*}
\left(\tilde{\mathcal{D}} f^h_\lambda\right)(x) =
|cx+d|^{-2\nu}(\tilde{\mathcal{D}}f)
\bigg(\frac{ax+b}{cx+d}\bigg)
\quad
\text{for all $f \in C^\infty(\mathbb{C}), 
h \in SL(2,\mathbb{R})$, 
and $x \in \mathbb{R}^\times$.}
\end{equation*}
\end{pquestion}

\begin{ppquestion}\label{quest:B''}
Find an explicit formula of 
conformally covariant differential operator 
$\mathcal{E}_\lambda(S^2) \to \mathcal{E}_\nu(S^1)$.
\end{ppquestion}

\begin{pppquestion}\label{quest:B'''}
Find an explicit expression of the element $\varphi(\mathbf{1}_{-\nu})$
for  any $\varphi \in  \operatorname{Hom}_\mathfrak{g} \left(
M(-\nu), M(-\lambda) \otimes M(-\lambda) \right)$,
where $\mathfrak{g} = \mathfrak{sl}(2, \mathbb{C})$.
\end{pppquestion}

An answer to Question \ref{quest:B''}
(and also in the case $S^{n-1} \subset S^n$
for arbitrary $n \geq 2$)
was first given by
Juhl \cite{Juhl}. 
In the flat model (Questions \ref{quest:B} and \ref{quest:B'}),
if $a : = \nu-\lambda \in \mathbb{N}$
then
\begin{equation*}
\widetilde{\mathcal C{}^{\lambda-\frac{1}{2}}_a}
\equiv\Rest{y=0} \circ C^{\lambda-\frac{1}{2}}_a\bigg(-\frac{\partial^2}{\partial x^2},
\frac{\partial}{\partial y}\bigg)
:\mathcal{E}_\lambda(\mathbb{R}^2) \to
\mathcal{E}_{\nu}(\mathbb{R}^1)
\end{equation*}
intertwines the $SL(2,\mathbb{R})$-action.
There have been several proofs for this 
(and also for more general cases) based on:

\begin{itemize}
\item Recurrence relations among coefficients of $\mathcal{D}$ (\cite{Juhl}),
\item F-method (\cite{K13, KOSS, KP}), and
\item Residue \text{formul\ae} of a meromorphic family of non-local
symmetry breaking operators \cite{Kob14a, KS}.
\end{itemize}

The holomorphic trick in Section \ref{sec:htrick} applied to this 
case gives yet another proof by using the Rankin--Cohen brackets
and the following proposition analogous to 
(and much simpler than) Proposition \ref{prop:Jacob}.

\begin{proposition}
For $a \in \mathbb{N}$, we have 
\begin{align*}
(1-z)^aP^{\lambda-1, -2\lambda-2a+1}_a
\bigg(\frac{3+z}{1-z}\bigg)
=(-1)^a
\frac{\left(\lambda+ [\frac{a}{2}] \right)_{[\frac{a+1}{2}]}}
{\left(\lambda - \frac{1}{2} \right)_{[\frac{a+1}{2}]}}
C^{\lambda-\frac{1}{2}}_{a}(z).
\end{align*}
Equivalently,
\begin{align*}
P^{\lambda-1, -2\lambda-2a+1}_a(x-\sqrt{-1}y, x+\sqrt{-1}y)
=(\sqrt{-1})^a 
\frac{\left(\lambda+ [\frac{a}{2}] \right)_{[\frac{a+1}{2}]}}
{\left(\lambda - \frac{1}{2} \right)_{[\frac{a+1}{2}]}}
C^{\lambda-\frac{1}{2}}_a (-x^2, y).
\end{align*}
\end{proposition}

\subsection*{Acknowledgments}
The first author warmly thanks Professor Vladimir Dobrev for his
hospitality during the tenth International Workshop: Lie Theory and its Applications
in Physics in Varna, Bulgaria, 17-23 June 2013. 
Authors were partially supported by CNRS, 
the FMSP program at 
the Graduate School of Mathematical Sciences
of the University of Tokyo, and 
Japan Society for the Promotion of Sciences through
Grant-in-Aid for Scientific Research (A)
(25247006) and Short-Term Fellowship (S13024).

\end{document}